\documentclass[10pt,twoside]{article}
\usepackage[left=25.4mm, right=25.4mm, top=25.4mm, bottom=25.4mm]{geometry}   % Margins
\usepackage[utf8]{inputenc}                              % UTF-8 input
\usepackage[english]{babel}                              % English languaje
\usepackage[hidelinks]{hyperref}                         % Hyperlinks
\usepackage{booktabs}                                    % Nice tables
\usepackage{amsmath, amsfonts, stmaryrd}                 % Math packages
\usepackage{caption, subcaption}                         % Captions
\usepackage{tikz-cd}                                     % Tikz
\usepackage{enumitem}                                    % Enumeration
\usepackage{xcolor, graphicx}                            % Figures
\usepackage[noblocks]{authblk}                           % Author and affiliation
\usepackage{titlesec}                                    % Title adjusts
\usepackage{fancyhdr}                                    % Headings
\usepackage{etoolbox}
\usepackage[noadjust]{cite}
\apptocmd{\thebibliography}{\setlength{\itemsep}{-3pt}}{}{}
\usepackage{amsthm}
\theoremstyle{plain}
\newtheorem{theorem}{Theorem}[section]

\newtheorem{proposition}[theorem]{Proposition}
\newtheorem{corollary}[theorem]{Corollary}

\theoremstyle{definition}
\newtheorem{definition}[theorem]{Definition}
\newtheorem{example}[theorem]{Example}

\theoremstyle{remark}

\newcommand{\R}{\mathbb R}
\newcommand{\C}{\mathbb C}
\newcommand{\Ll}{\mathbb L}
\newcommand{\Sph}{\mathbb S}
\newcommand{\G}{\mathcal G}
\newcommand{\J}{\mathcal J}
\newcommand{\N}{\mathcal N}
\newcommand{\F}{\mathcal F}

\newcommand{\Lie}{\mathcal L}

\newcommand{\xm}{\mathfrak X(M)}
\newcommand{\fm}{\mathfrak F(M)}
\newcommand{\dm}{\Omega^1 (M)}

\newcommand{\GTM}{\Gamma (\mathbb TM)}
\newcommand{\TM}{\mathbb TM}
\newcommand{\TTM}{TM\oplus T^*M}
\newcommand{\TpM}{\mathbb T_pM}

\newcommand{\var}{\varepsilon}

\makeatletter
\def\blfootnote{\gdef\@thefnmark{}\@footnotetext}
\makeatother

\title{\bf About generalized complex structures on $\mathbb S^6$
  \blfootnote{Partially supported by a doctoral research grant ``Concepción Arenal'' (Universidad de Cantabria and Gobierno de Cantabria).}
  \blfootnote{\em E-mail addresses:}
  \blfootnote{Fernando Etayo: \href{mailto:fernando.etayo@unican.es}{fernando.etayo@unican.es}}
  \blfootnote{Pablo Gómez-Nicolás: \href{mailto:pablo.gomeznicolas@unican.es}{pablo.gomeznicolas@unican.es}}
  \blfootnote{Rafael Santamaría: \href{mailto:rsans@unileon.es}{rsans@unileon.es}}
}
\author[*]{Fernando Etayo}
\author[*]{Pablo Gómez-Nicolás}
\author[$\dag$]{Rafael Santamaría}
\affil[*]{\footnotesize Departamento de Matemáticas, Estadística y Computación, Facultad de Ciencias, Universidad de Cantabria, Avda. de los Castros, s/n, 39071, Santander, Spain}
\affil[$\dag$]{\footnotesize Departamento de Matemáticas, Escuela de Ingenierías Industrial, Informática y Aeroespacial, Universidad de León, Campus de Vegazana, 24071, León, Spain}
\date{\small \today}

\begin{document}

% TITLE
\maketitle

% ABSTRACT
\begin{abstract}
\noindent We study the existence of generalized complex structures on the six-dimensional sphere $\mathbb S^6$. We work with the generalized tangent bundle $\mathbb T\mathbb S^6\to \mathbb S^6$ and define the integrability of generalized geometric structures in terms of the Dorfman bracket. Specifically, we prove that there is not a direct way to induce a generalized complex structure on $\mathbb S^6$ from its usual nearly Kähler structure inherited from the octonions product.
\end{abstract}

% CLASSIFICATION AND KEYWORDS


\thispagestyle{empty}
% SECTION: INTRODUCTION
\section{Introduction}
\label{SECTION:INTRODUCTION}

The problem of finding almost complex structures that are integrable on even-dimensional spheres $\Sph^{2n}$ was first presented by H. Hopf in \cite{HOPF1947}. Since then, multiple results concerning this topic have been presented: for example, A. Borel and J. P. Serre proved in \cite{BORELSERRE1953} that the only spheres admitting almost complex structures are $\Sph^2$ and $\Sph^6$. It is widely known that the two-dimensional sphere $\Sph^2$ admits an integrable complex structure, but there is not a definite result that determines whether the six-dimensional sphere $\Sph^6$ can be endowed with an integrable complex structure. This question is known as the ``Hopf problem''.

Multiple points of view have been adopted in order to solve the Hopf problem (see, for example, \cite{AGRICOLABAZZONIGOERTSCHESKONSTANTISROLLENSKE2018} for a historical point of view and \cite{AGRICOLABOROWKAFRIEDRICH2018} to know about the state of the art). One of these points of view concerns the search of partial results, studying the existence of complex structures on $\Sph^6$ that satisfy additional requirements. Following this line, one important result by A. Blanchard (see \cite{BLANCHARD1953}) shows that $\Sph^6$ does not admit any complex structure compatible with the Euclidean metric of $\R^7$. This result was popularized by C. LeBrun in \cite{LEBRUN1987}.

A question that has not been addressed yet, as far as we know, is the Hopf problem in generalized geometry. Generalized geometry was firstly introduced by N. Hitchin in \cite{HITCHIN2003}, and studied more deeply by M. Gualtieri in \cite{GUALTIERI2004}. This geometry field focuses on the generalized tangent bundle $\TM := TM\oplus T^*M\to M$ of a  manifold. The first researches about generalized geometry dealt with generalized almost complex structures (by M. Gualtieri, \cite{GUALTIERI2004}) and generalized almost paracomplex structures (by A. Wade, \cite{WADE2004}). Generalized almost complex (resp. paracomplex) structures were defined as bundle endomorphisms $\J\colon \TM\to \TM$ such that $\J^2 = -\mathcal Id$ (resp. $\J^2 = \mathcal Id$) and that met a compatibility condition with the canonical pairing $\G_0$, defined in Eq. \eqref{eq:canonicalmetric}. This compatibility condition was that $\G_0(\J u, v) = -\G_0(u, \J v)$ for every $u, v\in \TM$.

However, in later studies such as \cite{NANNICINI2006, NANNICINI2010} by A. Nannicini and \cite{IDAMANEA2017} by C. Ida and A. Manea this condition was omitted in order to study a broader range of interesting geometric structures. This point of view agrees with specialised text in vector bundles (see, for example, \cite{POOR1981} by W. A. Poor), and it allows the study of these geometric structures in a similar way than polynomial structures on a manifold. In this document, we consider both compatible and non-compatible generalized almost complex structures with $\G_0$, and name them strong and weak structures, respectively. It is worth remarking that the notion of weakness used throughout this document is not the same as in other references; for example, in \cite{ALDIGRANDINI2022} a different notion of weakness is introduced for any generalized polynomial structure.

The compatibility condition is especially important when studying the integrability of generalized polynomial structures. The integrability of a generalized almost complex or paracomplex structure is defined in terms of the Dorfman bracket, introduced in Eq. \eqref{eq:dorfmanbracket}; and it can be characterized using the generalized Nijenhuis map (see Eq. \eqref{eq:nijenhuis} for generalized almost complex structures). One interesting property of the generalized Nijenhuis map is that, as it is shown in Proposition \ref{prop:conditiontensornijenhuis} for generalized almost complex structures, it acts as a tensor when the structure is compatible with $\G_0$. In other words, even though the generalized Nijenhuis map associated to a weak generalized almost complex structure can be used to study its integrability (Proposition \ref{prop:nijenhuis}), it may not behave as a tensor (see the example after Proposition \ref{prop:conditiontensornijenhuis}). The problem in this case is that the integrability of the structure cannot be checked on local frames, as pointed out by M. Aldi, S. Da Silva and D. Grandini in \cite{ALDIDASILVAGRANDINI2024}, for example. However, we are able to obtain necessary local conditions for a weak generalized almost complex structure to be integrable (see Theorem \ref{thm:integrabilitycoord}).

On the other hand, weak structures allow us to define a greater number of interesting geometries. For example, when the manifold $M$ is endowed with an almost Hermitian structure $(J,g)$, numerous polynomial generalized structures can be induced from $J$, $g$ and its fundamental form $\omega(\cdot, \cdot) = g(J\cdot, \cdot)$ (see, for example, \cite{ETAYOGOMEZNICOLASSANTAMARIA2022}). In particular, the almost complex structure $J$ induces two generalized almost complex structures, defined as $\J_{\lambda,J}(X + \xi) = JX + \lambda J^*\xi$ for $\lambda\in \{1,-1\}$; the fundamental form $\omega$ defines the generalized almost complex structure $\J_\omega(X + \xi) = -\sharp_\omega\xi + \flat_\omega X$; and the metric $g$ produces the generalized almost complex structure $\J_g(X + \xi) = -\sharp_g\xi + \flat_g X$. It can also be checked that any spherical combination $a\J_{1,J} + b\J_g + c\J_\omega$, with $a,b,c\in \fm$ such that $a^2 + b^2 + c^2 = 1$, is also a generalized almost complex structure (see Corollary \ref{cor:sphericalcombination}).

M. Gualtieri proved that, for an almost complex structure $J$, the induced generalized almost complex structure $\J_{-1,J}$ is integrable if and only if $J$ is integrable in the classical way. Also, for an almost symplectic structure $\omega$, we have that $\J_\omega$ is an integrable generalized complex structure if and only if $\omega$ is symplectic. Therefore, any manifold that admits a complex structure or a symplectic structure admits an integrable generalized complex structure, and hence admits generalized complex structures. This is the case of the two-dimensional sphere $\Sph^2$, which admits a complex and a symplectic structure. In the case of the six-dimensional sphere $\Sph^6$, since $H^2(\Sph^6,\R) = 0$ we know that it does not admit a symplectic structure. In respect of strong generalized complex structures, even though it is not known if there is an integrable complex structure on $\Sph^6$ we can study the problem on the bundle $\mathbb T\Sph^6$ and determine if there are some particular integrable generalized complex structures. The main result that is proved in this part of the paper is summarized in the following statement:

\begin{theorem}
\label{thm:s6structures}
There are no weak generalized complex structures on $\Sph^6$ that can be written as a spherical combination of the induced generalized almost complex structures $\J_{1,J}, \J_g, \J_\omega$ from Eqs. {\normalfont (\ref{eq:complexlambdaJ}, \ref{eq:complexomega}, \ref{eq:complexg})}, where $(J, g)$ is the nearly Kähler structure on $\Sph^6$ inherited from the pure octonions product.
\end{theorem}

The paper is organized as follows:

In the first part of the paper, we introduce the basic concepts and fix the notation that is necessary to work with the generalized tangent bundle $\TM$ of a manifold. We distinguish between weak and strong generalized almost complex structures regarding their compatibility with the canonical metric that arises in $\TM$. We also take the definition of integrability of an endomorphism and analyse the behaviour of the generalized Nijenhuis map of a generalized almost complex structure, depending on whether the structure is weak or strong. Lastly, in Theorem \ref{thm:integrabilitycoord} we find necessary conditions for the integrability of a generalized almost complex structure working in local coordinates.

In the second part of the paper, we consider the particular case of the six-dimensional sphere $\Sph^6$ endowed with its well-known nearly Kähler structure, inherited from the pure octonions product on $\R^7$. Following the spirit of other partial results, our aim is to study the existence of a specific type of generalized complex structures. Working in local coordinates with the nearly Kähler structure on $\Sph^6$, and using spherical combinations of the induced generalized almost complex structures, we show that the necessary conditions found in Theorem \ref{thm:integrabilitycoord} are not fulfilled. In this way, we prove the Theorem \ref{thm:s6structures}.

It is worthwhile remarking that the problem of finding generalized paracomplex structures on $\Sph^6$ is extremely easy, since in every manifold it can be defined the canonical strong generalized paracomplex structure $\F_0(X + \xi) = X - \xi$, as A. Wade introduced in \cite{WADE2004}. Therefore, we just mention it and focus on generalized complex structures.

% SECTION: PRELIMINARIES
\section{Generalized almost complex structures and integrability}
\label{SECTION:PRELIMINARIES}

First, we recall the basic notations we are going to use. Throughout the paper, manifolds of dimension $n$ and tensor fields are smooth (i.e., $\mathcal C^\infty$), and Einstein summation convention will be used when we work in local coordinates. The \emph{generalized tangent bundle} of $M$ is defined as the Whitney sum of its tangent and cotangent bundle, namely $\TM := \TTM\to M$. On each point, the fibre of this vector bundle is the direct sum of its tangent and cotangent space, that is, $\TpM = T_pM\oplus T^*_pM$ for $p\in M$. The smooth sections $\GTM$ of this bundle can be written as sums of a vector field and a differential form on $M$; that is to say, $\GTM = \xm\oplus \dm$, and for any $u\in \GTM$ there are some unique $X\in \xm$ and $\xi\in \dm$ such that $u = X + \xi$.

In this work, we focus mainly on bundle endomorphisms $\J\colon \TM\to \TM$ such that each fibre of $\TM$ is mapped to itself. The following matrix notation is used to describe these morphisms. Any bundle endomorphism $\J$ of $\TM$ can be written as
	\begin{equation}
	  \J = \left(
		\begin{array}{cc}
		  A  & B  \\
		  C  & D
		\end{array}
	  \right),
	  \label{eq:matrixendomorphism}
	\end{equation}
for some bundle morphisms $A\colon TM\to TM$, $B\colon T^*M\to TM$, $C\colon TM\to T^*M$ and $D\colon T^*M\to T^*M$. This means that for any $X + \xi\in \TM$ the following expression can be written:
	\begin{equation*}
	  \J(X + \xi) = \left(
	    \begin{array}{cc}
	      A  & B  \\
	      C  & D
	    \end{array}
	  \right) \left(
	    \begin{array}{c}
	      X    \\
	      \xi
	    \end{array}
	  \right) = \left(
	    \begin{array}{c}
	      AX + B\xi  \\
	      CX + D\xi
	    \end{array}
	  \right) = (AX + B\xi) + (CX + D\xi)\in \TTM.
	\end{equation*}

On this vector bundle, there is a \emph{canonical generalized metric} that emerges without the need of adding any further structure to $M$. This morphism $\G_0\in \Gamma((\TM)^*\otimes (\TM)^*)$ is defined such that for each $X, Y\in T_pM$ and $\xi, \eta\in T^*_pM$,
    \begin{equation}
	  \G_0(X + \xi, Y + \eta) := \frac{1}{2}(\xi(Y) + \eta(X)).
	  \label{eq:canonicalmetric}
	\end{equation}
It is immediate to check that $\G_0$ is symmetric, nondegenerate, and its signature is $(n, n)$.

Although generalized almost complex structures were originally required to be isometric with respect to the canonical metric $\G_0$, later studies suggest that this condition may be omitted. In order to study a wider range of interesting geometric structures, but not to forget about this relevant property, we give the following definition.

\begin{definition}
\label{def:generalizedalmostcomplex}
A \emph{weak generalized almost complex structure} is a bundle endomorphism of the generalized tangent bundle $\J\colon \TM\to \TM$ such that $\J^2 = -\mathcal Id$. If for every $X + \xi, Y + \eta\in \TpM$, such an endomorphism $\J$ also satisfies
	\begin{equation}
	\G_0(\J(X + \xi),Y + \eta) = -\G_0(X + \xi, \J(Y + \eta)),
	\label{eq:strongcondition}
	\end{equation}
it will be called a \emph{strong generalized almost complex structure}.
\end{definition}

Some interesting examples of generalized almost complex structures appear when the base manifold is endowed with some geometric structure. These examples have been previously studied by other authors (see, for example, \cite{GUALTIERI2011, NANNICINI2010}).

\begin{example}
\label{ex:complexlambdaJ}
Let $(M,J)$ be an almost complex manifold, and $\lambda\in \{1,-1\}$. Then, the generalized almost complex structure
    \begin{equation}
      \J_{\lambda, J} =
      \left(
        \begin{array}{cc}
          J  & 0  \\
          0  & \lambda J^*
        \end{array}
      \right),
      \label{eq:complexlambdaJ}
    \end{equation}
where the \emph{dual structure} $J^*$ is defined as $(J^*\xi)X = \xi(JX)$ for $\xi\in T_p^*M$ and $X\in T_pM$, is weak for $\lambda = 1$ and strong for $\lambda = -1$.
\end{example}

\begin{example}
\label{ex:complexomega}
If we consider an almost symplectic manifold $(M,\omega)$, the following endomorphism is a strong generalized almost complex structure:
    \begin{equation}
      \J_{\omega} =
      \left(
        \begin{array}{cc}
          0        & -\sharp_{\omega}  \\
          \flat_{\omega}  & 0
        \end{array}
      \right).
      \label{eq:complexomega}
    \end{equation}
In the last expression, the \emph{musical isomorphisms} associated to $\omega$, $\flat_\omega\colon TM\to T^*M$ and $\sharp_\omega\colon T^*M\to TM$, are obtained in such a manner that $(\flat_\omega X)Y = \omega(X, Y)$ for every $X,Y\in T_pM$, and $\sharp_\omega = \flat_\omega^{-1}$.
\end{example}

\begin{example}
\label{ex:complexg}
Now, we take a (pseudo-)Riemannian manifold $(M,g)$. Then, the following generalized almost complex structure is weak:
    \begin{equation}
      \J_g =
      \left(
        \begin{array}{cc}
          0        & -\sharp_g  \\
          \flat_g  & 0
        \end{array}
      \right).
      \label{eq:complexg}
    \end{equation}
The musical isomorphisms associated to $g$ are defined analogously to the ones in Example \ref{ex:complexomega}.
\end{example}

It must be noted that, even though the structures $\J_{1,J}$ and $\J_g$ are considered weak generalized almost complex structures, they both are compatible with $\G_0$ in a different way than strong structures:
    \begin{equation*}
	  \G_0(\J_{1,J}(X + \xi), Y + \eta) = \G_0(X + \xi, \J_{1,J}(Y + \eta)),
	\end{equation*}
    \begin{equation*}
	  \G_0(\J_g(X + \xi), Y + \eta) = \G_0(X + \xi, \J_g(Y + \eta)).
	\end{equation*}
However, we will use the ``strong'' adjective only for structures that fulfil the condition given in Eq. \eqref{eq:strongcondition}.

For the sake of introducing the concept of integrability of a generalized almost complex structure, we must define a bracket product in $\TM$. The most used bracket in generalized geometry is the \emph{Dorfman bracket}, because it makes the structure $\J_{-1,J}$ from Eq. \eqref{eq:complexlambdaJ} integrable if and only if $J$ is integrable in the usual way; and the structure $\J_\omega$ from Eq. \eqref{eq:complexomega} integrable if and only if $\omega$ is symplectic. This bracket is defined for any $X, Y\in \xm$ and $\xi, \eta\in \dm$ as
    \begin{equation}
	  \llbracket X + \xi, Y + \eta\rrbracket = [X, Y] + \Lie_X \eta - i_Y d\xi.
	  \label{eq:dorfmanbracket}
	\end{equation}
In the previous expression, $[\cdot,\cdot]$ denotes the Lie bracket of vector fields; $\Lie_X$ is the Lie derivative with respect to $X$; and $i_Y$ is the contraction with $Y$. It is immediate to see that, unlike the Lie bracket, $\llbracket \cdot, \cdot\rrbracket$ is not skew-symmetric.

\begin{definition}
\label{def:integrability}
A weak (resp. strong) generalized almost complex structure $\J$ is said to be \emph{integrable} when the distribution $\Ll^{1,0}_\J = \{u - i\J u\colon u\in \TM\}\subset \TM\otimes \C$ is involutive with respect to the Dorfman bracket (in other words, if $\llbracket u - i\J u, v - i\J v\rrbracket\in \Gamma(\Ll^{1,0}_\J)$ for every $u,v\in \GTM$). Then, $\J$ is called a \emph{weak} (resp. \emph{strong}) \emph{generalized complex structure}.
\end{definition}

The integrability of a generalized almost complex structure can easily be characterized in terms of the \emph{generalized Nijenhuis map} associated to $\J$, which is defined for $u,v\in \GTM$ as
	\begin{equation}
	  \N_\J(u, v) = \llbracket\J u, \J v\rrbracket - \J(\llbracket\J u, v\rrbracket + \llbracket u, \J v\rrbracket) - \llbracket u, v\rrbracket.
	  \label{eq:nijenhuis}
	\end{equation}

\begin{proposition}
\label{prop:nijenhuis}
A weak generalized almost complex structure $\J$ is integrable if and only if $\N_\J\equiv 0$.
\end{proposition}

\begin{proof}
The Dorfman bracket can be extended to $\TM\otimes \C$ in such a way that $\llbracket iu, v\rrbracket = \llbracket u, iv\rrbracket = i\llbracket u, v\rrbracket$, so it is a $\C$-bilinear map. If $\llbracket u - i\J u, v - i\J v\rrbracket$ is computed for any $u,v\in \GTM$, we have
	\begin{equation*}
	  \llbracket u - i\J u, v - i\J v\rrbracket = \llbracket u, v\rrbracket - \llbracket \J u, \J v\rrbracket - i(\llbracket\J u, v\rrbracket + \llbracket u, \J v\rrbracket).
	\end{equation*}
Then, $\J$ will be integrable iff $\llbracket u - i\J u, v - i\J v\rrbracket$ can be written as $w - i\J w$ for some $w\in \GTM$. From the last expression, it is inferred that $w = \llbracket u, v\rrbracket - \llbracket \J u, \J v\rrbracket$ and $\J w = \llbracket\J u, v\rrbracket + \llbracket u, \J v\rrbracket$. Therefore, $\J$ will be integrable iff $\llbracket\J u, v\rrbracket + \llbracket u, \J v\rrbracket = \J(\llbracket u, v\rrbracket - \llbracket \J u, \J v\rrbracket)$ or, equivalently,
	\begin{equation*}
	  \llbracket\J u, \J v\rrbracket - \llbracket u, v\rrbracket - \J(\llbracket\J u, v\rrbracket + \llbracket u, \J v\rrbracket) = \N_\J(u, v) = 0,
	\end{equation*}
as we wanted to prove.
\end{proof}

It is worth remarking that if $\J$ is a strong generalized almost complex structure then $\Ll^{1,0}_\J$ is \emph{isotropic} with respect to the canonical generalized metric $\G_0$ (that is, $\Ll^{1,0}_\J\subset (\Ll^{1,0}_\J)^\perp$): for any $u,v\in \GTM$, if $\J$ is strong then
	\begin{equation*}
	  \G_0(u - i\J u, v - i\J v) = \G_0(u, v) - \G_0(\J u, \J v) - i(\G_0(\J u, v) + \G_0(u, \J v)) = 0.
	\end{equation*}
In fact, the following two conditions are equivalent for strong generalized almost complex structures \cite[Prop. 3.27]{GUALTIERI2004}:
\begin{itemize}

\item $\N_\J \equiv 0$.

\item $\Ll^{1,0}_\J$ is a maximal isotropic subbundle of $\TM\otimes \C$ which is involutive with respect to the Dorfman bracket.

\end{itemize}
As we have just seen in Proposition \ref{prop:nijenhuis}, the equivalence between $\N_\J\equiv 0$ and the fact that $\Ll^{1,0}_\J$ is involutive under the Dorfman bracket remains true in the case of weak generalized almost complex structures. In other works, such as \cite{IDAMANEA2017} by C. Ida and A. Manea, and \cite{HUMORARUSVOBODA2019} by S. Hu, R. Moraru and D. Svoboda, non-isotropic structures are also studied regarding its integrability.

It is of utmost importance to comment on some properties of the generalized Nijenhuis map and the differences with respect to the usual Nijenhuis tensor of two vector fields. In other works where strong generalized structures are studied, such as \cite{CORTESDAVID2021} by V. Cortés and L. David, it has been checked that the generalized Nijenhuis map behaves as a tensor on a manifold, as it is shown in the following statement:

\begin{proposition}[{\cite[Lemma 9]{CORTESDAVID2021}}]
\label{prop:conditiontensornijenhuis}
The Nijenhuis generalized map $\N_\J(u, v)$ associated to a strong generalized almost complex structure is $\fm$-linear in both $u$ and $v$.
\end{proposition}

Then, the generalized Nijenhuis map of a strong generalized almost complex structure can be called its \emph{generalized Nijenhuis tensor}. However, this is not the case for weak generalized almost complex structures: $\N_\J$ is not always $\fm$-linear. To see this, take any almost Hermitian manifold $(M, J, g)$. This geometric structure induces the following generalized almost complex structure:
	\begin{equation*}
	  \J = \left(
        \begin{array}{cc}
          J  & \sharp_g  \\
          0  & J^* 
        \end{array}
      \right).
	\end{equation*}
It is easy to check that $\J^2 = -\mathcal Id$: knowing that $J\sharp_g = -\sharp_g J^*$ (\cite[Prop. 2.8]{ETAYOGOMEZNICOLASSANTAMARIA2022}),
\begin{equation*}
	  \J^2 = \left(
        \begin{array}{cc}
          J  & \sharp_g  \\
          0  & J^* 
        \end{array}
      \right)\left(
        \begin{array}{cc}
          J  & \sharp_g  \\
          0  & J^* 
        \end{array}
      \right) = \left(
        \begin{array}{cc}
          J^2  & J\sharp_g + \sharp_g J^*  \\
          0    & (J^*)^2 
        \end{array}
      \right) = \left(
        \begin{array}{cc}
          -Id  & 0  \\
          0    & -Id 
        \end{array}
      \right).
	\end{equation*}
Also, this generalized structure is weak (see, for example, \cite[Prop. 4.9]{ETAYOGOMEZNICOLASSANTAMARIA2024}): a direct calculation shows that
	\begin{equation*}
	  \begin{split}
	    \G_0(\J(X + \xi), \J(Y + \eta)) &= \G_0(JX + \sharp_g\xi + J^*\xi, JY + \sharp_g\eta + J^*\eta)  \\
	   																 &= \frac{1}{2}((J^*\xi)(JY + \sharp_g\eta) + (J^*\eta)(JX + \sharp_g\xi))  \\
	  															  	 &= \frac{1}{2}(g(\sharp_g \xi, J\sharp_g \eta) + g(\sharp_g \eta, J\sharp_g \xi) + \xi(J^2 Y) + \eta(J^2 X))  \\
	  															  	 &= - \mathcal G_0(X + \xi, Y + \eta),
	  \end{split}
	\end{equation*}
thus proving that it is not strong. In order to see that its generalized Nijenhuis map does not behave as a tensor, we may calculate $\N_\J(\xi, \eta)$ for any $\xi,\eta\in \dm$:	
	\begin{equation*}
	  \begin{split}
	    \N_\J(\xi, \eta) &= \llbracket \J\xi, \J\eta\rrbracket - \J(\llbracket \J\xi, \eta\rrbracket + \llbracket \xi, \J\eta\rrbracket) - \llbracket \xi, \eta\rrbracket  \\
	    									&= \llbracket \sharp_g\xi + J^*\xi, \sharp_g\eta + J^*\eta\rrbracket - \J(\llbracket \sharp_g\xi + J^*\xi, \eta\rrbracket + \llbracket \xi, \sharp_g\eta + J^*\eta\rrbracket)  \\
											&= [\sharp_g\xi, \sharp_g\eta] + \Lie_{\sharp_g\xi}(J^*\eta) - i_{\sharp_g\eta}d(J^*\xi) - \J(\Lie_{\sharp_g\xi}\eta - i_{\sharp_g\eta}d\xi)  \\
	    									&= [\sharp_g\xi, \sharp_g\eta] - \sharp_g\Lie_{\sharp_g\xi}\eta + \sharp_g i_{\sharp_g\eta}d\xi + \Lie_{\sharp_g\xi}(J^*\eta) - i_{\sharp_g\eta}d(J^*\xi) - J^*\Lie_{\sharp_g\xi}\eta + J^*i_{\sharp_g\eta}d\xi.
	  \end{split}
	\end{equation*}
If we now multiply $\xi$ by a function $f\in \fm$, it is clear that the map $\N_\J$ is not $\fm$-linear:
	\begin{equation*}
	  \begin{split}
	    \N_\J(f\xi, \eta) &= [f\sharp_g\xi, \sharp_g\eta] - \sharp_g\Lie_{f\sharp_g\xi}\eta + \sharp_g i_{\sharp_g\eta}d(f\xi) + \Lie_{f\sharp_g\xi}(J^*\eta) - i_{\sharp_g\eta}d(fJ^*\xi) - J^*\Lie_{f\sharp_g\xi}\eta + J^*i_{\sharp_g\eta}d(f\xi)  \\
	    				  &= f\N_\J(\xi, \eta) - (\sharp_g\eta)(f)\sharp_g\xi - \sharp_g df \eta(\sharp_g\xi) + (\sharp_g\eta)(f)\sharp_g\xi - \sharp_g df \xi(\sharp_g\eta)  \\
	    				  &\phantom{=} + df(J^*\eta)(\sharp_g\xi) - (\sharp_g\eta)(f)J^*\xi + df(J^*\xi)(\sharp_g\eta) - J^*df \eta(\sharp_g\xi) + (\sharp_g\eta)(f)J^*\xi - J^*df \xi(\sharp_g\eta)  \\
	    									 &= f\N_\J(\xi, \eta) - 2g(\sharp_g\xi, \sharp_g\eta)\sharp_g df - 2g(\sharp_g\xi, \sharp_g\eta)J^*df.
	  \end{split}
	\end{equation*}

The next result focuses on conditions that a structure must fulfil in order to be integrable.

\begin{theorem}
\label{thm:integrabilitycoord}
Let $\J\colon\TM\to \TM$ be a weak generalized almost complex structure that is written in matrix form as in Eq. \eqref{eq:matrixendomorphism}. Let us take local coordinates $(U, (x^1, \dots, x^n))$ such that
	\begin{equation*}
	  A\partial_i = A_i^j\partial_j, \enspace\enspace B dx^i = B^{ij}\partial_j, \enspace\enspace C\partial_i = C_{ij} dx^j, \enspace\enspace D dx^i = D^i_j dx^j.
	\end{equation*}
If the generalized structure $\J$ is integrable, then the following conditions are met for $i,j,l = 1,\dots,n$:
	\begin{align}
	  A^k_i\frac{\partial A^l_j}{\partial x^k} - A^k_j\frac{\partial A^l_i}{\partial x^k} + A^l_k\left(\frac{\partial A^k_i}{\partial x^j} - \frac{\partial A^k_j}{\partial x^i}\right) - B^{kl}\left(\frac{\partial C_{jk}}{\partial x^i} - \frac{\partial C_{ik}}{\partial x^j} + \frac{\partial C_{ij}}{\partial x^k}\right)
	  &= 0,  \label{eq:integrability1} \\
	  A^k_i\frac{\partial C_{jl}}{\partial x^k} + C_{jk}\frac{\partial A^k_i}{\partial x^l} + A^k_j\left(\frac{\partial C_{ik}}{\partial x^l} - \frac{\partial C_{il}}{\partial x^k}\right) + C_{kl}\left(\frac{\partial A^k_i}{\partial x^j} - \frac{\partial A^k_j}{\partial x^i}\right) - D^k_l\left(\frac{\partial C_{jk}}{\partial x^i} - \frac{\partial C_{ik}}{\partial x^j} + \frac{\partial C_{ij}}{\partial x^k}\right)
	  &= 0,  \label{eq:integrability2}  \\
	  A^k_i\frac{\partial B^{jl}}{\partial x^k} - B^{jk}\frac{\partial A^l_i}{\partial x^k} - A^l_k\frac{\partial B^{jk}}{\partial x^i} - B^{kl}\left(\frac{\partial A^j_i}{\partial x^k} + \frac{\partial D^j_k}{\partial x^i}\right)
	  &= 0,  \label{eq:integrability3}  \\
		A^k_i\frac{\partial D^j_l}{\partial x^k} + D^j_k\frac{\partial A^k_i}{\partial x^l} + B^{jk}\left(\frac{\partial C_{ik}}{\partial x^l} - \frac{\partial C_{il}}{\partial x^k}\right) - C_{kl}\frac{\partial B^{jk}}{\partial x^i} - D^k_l\left(\frac{\partial A^j_i}{\partial x^k} + \frac{\partial D^j_k}{\partial x^i}\right)
	  &= 0,  \label{eq:integrability4}  \\
	  B^{ik}\frac{\partial A^l_j}{\partial x^k} - A^k_j\frac{\partial B^{il}}{\partial x^k} + A^l_k\frac{\partial B^{ik}}{\partial x^j} + B^{kl}\left(\frac{\partial D^i_k}{\partial x^j} - \frac{\partial D^i_j}{\partial x^k}\right)
	  &= 0,  \label{eq:integrability5}  \\
	  B^{ik}\frac{\partial C_{jl}}{\partial x^k} + C_{jk}\frac{\partial B^{ik}}{\partial x^l} + A^k_j\left(\frac{\partial D^i_k}{\partial x^l} - \frac{\partial D^i_l}{\partial x^k}\right) + C_{kl}\frac{\partial B^{ik}}{\partial x^j} + D^k_l\left(\frac{\partial D^i_k}{\partial x^j} - \frac{\partial D^i_j}{\partial x^k}\right)
	  &= 0,  \label{eq:integrability6} \\
	  B^{ik}\frac{\partial B^{jl}}{\partial x^k} - B^{jk}\frac{\partial B^{il}}{\partial x^k} - B^{kl}\frac{\partial B^{ij}}{\partial x^k}
	  &= 0,  \label{eq:integrability7} \\
	  B^{ik}\frac{\partial D^j_l}{\partial x^k} + D^j_k\frac{\partial B^{ik}}{\partial x^l} + B^{jk}\left(\frac{\partial D^i_k}{\partial x^l} - \frac{\partial D^i_l}{\partial x^k}\right) - D^k_l\frac{\partial B^{ij}}{\partial x^k}
	  &= 0.  \label{eq:integrability8}
	\end{align}
\end{theorem}

\begin{proof}
To obtain Eqs. (\ref{eq:integrability1}-\ref{eq:integrability8}), one must calculate $\N_\J(X, Y)$, $\N_\J(X, \eta)$, $\N_\J(\xi, Y)$, $\N_\J(\xi, \eta)$ for any vector fields and differential forms $X, Y\in \xm$ and $\xi, \eta\in \dm$. We compute $\N_\J(X, Y)$ explicitly:
	\begin{equation*}
	  \begin{split}
	    \N_\J(X, Y) &= \llbracket\J X, \J Y\rrbracket - \J(\llbracket\J X, Y\rrbracket + \llbracket X, \J Y\rrbracket) - \llbracket X, Y\rrbracket  \\
	    			&= \llbracket AX + CX, AY + CY\rrbracket - \J(\llbracket AX + CX, Y\rrbracket + \llbracket X, AY + CY\rrbracket) - \llbracket X, Y\rrbracket  \\
	    			&= [AX, AY] + \Lie_{AX}(CY) - i_{AY}d(CX) - \J([AX, Y] - i_Yd(CX) + [X, AY] + \Lie_X(CY)) - [X, Y]  \\
	    			&= [AX, AY] - A([AX, Y] + [X, AY]) + B(i_Yd(CX) - \Lie_X(CY)) - [X, Y]  \\
	    			&\phantom{=} + \Lie_{AX}(CY) - i_{AY}d(CX) - C([AX, Y] + [X, AY]) + D(i_Yd(CX) - \Lie_X(CY)).
	  \end{split}
	\end{equation*}	
The vector fields $X, Y$ can now be replaced by the coordinate fields $X = \partial_i$, $Y = \partial_j$. Then, we expand the vector field and differential form components of $\N_\J(\partial_i, \partial_j)$ separately, obtaining two different expressions:
	\begin{equation*}
	  \begin{split}
	    \left.\N_\J(\partial_i, \partial_j)\right|_{\mathfrak X(U)} &= [A\partial_i, A\partial_j] - A([A\partial_i, \partial_j] + [\partial_i, A\partial_j]) + B(i_{\partial_j} d(C\partial_i) - \Lie_{\partial_i}(C\partial_j)) - [\partial_i, \partial_j]  \\
	    															&= [A_i^k\partial_k, A_j^l\partial_l] - A([A_i^k\partial_k, \partial_j] + [\partial_i, A_j^k\partial_k]) + B(i_{\partial_j} d(C_{ik} dx^k) - \Lie_{\partial_i}(C_{jk} dx^k))  \\
	    															&= \left(A_i^k\frac{\partial A_j^l}{\partial x^k} - A_j^k\frac{\partial A_i^l}{\partial x^k}\right)\partial_l - A\left(\frac{\partial A_j^k}{\partial x^i} - \frac{\partial A_i^k}{\partial x^j}\right)\partial_k + B\left(\frac{\partial C_{ik}}{\partial x^j} - \frac{\partial C_{ij}}{\partial x^k} - \frac{\partial C_{jk}}{\partial x^i}\right)dx^k  \\
	    															&= \left[A^k_i\frac{\partial A^l_j}{\partial x^k} - A^k_j\frac{\partial A^l_i}{\partial x^k} + A^l_k\left(\frac{\partial A^k_i}{\partial x^j} - \frac{\partial A^k_j}{\partial x^i}\right) - B^{kl}\left(\frac{\partial C_{jk}}{\partial x^i} - \frac{\partial C_{ik}}{\partial x^j} + \frac{\partial C_{ij}}{\partial x^k}\right)\right]\partial_l,
	  \end{split}
	\end{equation*}
	\begin{equation*}
	  \begin{split}
	    \left.\N_\J(\partial_i, \partial_j)\right|_{\Omega^1(U)} &= \Lie_{A\partial_i}(C\partial_j) - i_{A\partial_j} d(C\partial_i) - C([A\partial_i, \partial_j] + [\partial_i, A\partial_j]) + D(i_{\partial_j} d(C\partial_i) - \Lie_{\partial_i}(C\partial_j))  \\
	    														 &= \Lie_{A_i^k\partial_k}(C_{jl}dx^l) - i_{A_j^k\partial_k} d(C_{il}dx^l) - C([A_i^k\partial_k, \partial_j] + [\partial_i, A_j^k\partial_k])  \\
	    														 &\phantom{=} + D(i_{\partial_j} d(C_{ik} dx^k) - \Lie_{\partial_i}(C_{jk} dx^k))  \\
	    														 &= \left(A^k_i\frac{\partial C_{jl}}{\partial x^k} + C_{jk}\frac{\partial A^k_i}{\partial x^l}\right)dx^l + A^k_j\left(\frac{\partial C_{ik}}{\partial x^l} - \frac{\partial C_{il}}{\partial x^k}\right)dx^l - C\left(\frac{\partial A_j^k}{\partial x^i} - \frac{\partial A_i^k}{\partial x^j}\right)\partial_k  \\
	    														 &\phantom{=} + D\left(\frac{\partial C_{ik}}{\partial x^j} - \frac{\partial C_{ij}}{\partial x^k} - \frac{\partial C_{jk}}{\partial x^i}\right)dx^k  \\
	    														 &= \left[A^k_i\frac{\partial C_{jl}}{\partial x^k} + C_{jk}\frac{\partial A^k_i}{\partial x^l} + A^k_j\left(\frac{\partial C_{ik}}{\partial x^l} - \frac{\partial C_{il}}{\partial x^k}\right) + C_{kl}\left(\frac{\partial A^k_i}{\partial x^j} - \frac{\partial A^k_j}{\partial x^i}\right)\right.  \\
	    														 &\phantom{=} \left.- D^k_l\left(\frac{\partial C_{ij}}{\partial x^k} - \frac{\partial C_{ik}}{\partial x^j} + \frac{\partial C_{jk}}{\partial x^i}\right)\right]dx^l.
	  \end{split}
	\end{equation*}
Therefore, it is proved that $\N_\J(\partial_i, \partial_j) = 0$ if and only if Eqs. (\ref{eq:integrability1}, \ref{eq:integrability2}) hold true. 

Similar calculations can be done for $\N_\J(\partial_i, dx^j)$, $\N_\J(dx^i, \partial_j)$ and $\N_\J(dx^i, dx^j)$, splitting each of them into its vector field and differential form components and obtaining the conditions in Eqs. (\ref{eq:integrability3}-\ref{eq:integrability8}).
\end{proof}

Because of the $\fm$-linearity of the generalized Nijenhuis tensor of a strong generalized almost complex structure, it is immediate to check that the converse implication is also true when the structure is strong. Then, we have the following proposition.

\begin{proposition}
\label{prop:strongintegrabilitycoord}
If the generalized almost complex structure $\J$ in Theorem \ref{thm:integrabilitycoord} is strong, then $\J$ is integrable if and only if Eqs. {\normalfont (\ref{eq:integrability1}-\ref{eq:integrability8})} hold true.
\end{proposition}

The above result can be compared with some of the generalized structures that have been studied in other works.

\begin{example}[{\cite[Ex. 4.21]{GUALTIERI2004}}]
\label{ex:integrabilitylambdaJ}
It is well-know that, if $(M,J)$ is an almost complex manifold, the strong generalized almost complex structure $\J_{-1,J}$ from Eq. \eqref{eq:complexlambdaJ} is integrable if and only if $J$ is integrable. In this case, we have $A = J$, $B = 0$, $C = 0$ and $D = -J^*$. Therefore, the only non-trivial expressions from Theorem \ref{thm:integrabilitycoord} are Eqs. (\ref{eq:integrability1}, \ref{eq:integrability4}, \ref{eq:integrability6}). Using local coordinates, $J\partial_i = J_i^k\partial_k$ and $J^*dx^i = J^i_k dx^k$. The Nijenhuis tensor of $J$, in coordinates, is equal to
	\begin{equation*}
	  N_J(\partial_i, \partial_j) = \left[J^k_i\frac{\partial J^l_j}{\partial x^k} - J^k_j\frac{\partial J^l_i}{\partial x^k} + J^l_k\left(\frac{\partial J^k_i}{\partial x^j} - \frac{\partial J^k_j}{\partial x^i}\right)\right]\partial_l.
	\end{equation*}
Then, Eq. \eqref{eq:integrability1} holds true if and only if $N_J \equiv 0$; in other words, if and only if $J$ is integrable. In respect of Eqs. (\ref{eq:integrability4}, \ref{eq:integrability6}), one must also use the fact that $J^2 = -Id$; taking local coordinates, $J^i_k J^k_j = \delta^i_j$ and, consequently,
	\begin{equation*}
	  J^i_k\frac{\partial J^k_j}{\partial x^l} = - J^k_j\frac{\partial J^i_k}{\partial x^l}.
	\end{equation*}
Thus, it is easy to check that, as expected, these two equations are true if and only if $J$ is integrable.
\end{example}

\begin{example}[{\cite[Ex. 4.20]{GUALTIERI2004}}]
\label{ex:integrabilityomega}
The other typical example of a generalized complex structure is induced from an almost symplectic manifold $(M,\omega)$. The generalized almost complex structure $\J_\omega$ from Eq. \eqref{eq:complexomega} is integrable if and only if $\omega$ is symplectic (i.e., $d\omega = 0$). Using its matrix form we have $A = 0$, $B = -\sharp_\omega$, $C = \flat_\omega$ and $D = 0$. Then, Eqs. (\ref{eq:integrability1}, \ref{eq:integrability4}, \ref{eq:integrability6}, \ref{eq:integrability7}) must be studied because of their non-triviality. Taking local coordinates, we have that $\flat_\omega\partial_i = \omega_{ij}dx^j$ and $-\sharp_\omega dx^i = -\omega^{ij}\partial_j$, where $\omega^{ik}\omega_{kj} = \delta^i_j$ because $\sharp_\omega = \flat_\omega^{-1}$. In coordinates, the local representation of the exterior derivative of $\omega$ is given by
	\begin{equation*}
	  (d\omega)_{ijk} = \frac{\partial\omega_{jk}}{\partial x^i} - \frac{\partial\omega_{ik}}{\partial x^j} + \frac{\partial\omega_{ij}}{\partial x^k}.
	\end{equation*}
Therefore, Eq. \eqref{eq:integrability1} is true if and only if $d\omega = 0$. Regarding Eqs. (\ref{eq:integrability4}, \ref{eq:integrability6}), in a similar way to Example \ref{ex:integrabilitylambdaJ}, knowing that
	\begin{equation*}
	  \omega^{ik}\frac{\partial \omega_{kj}}{\partial x^l} = - \omega_{kj}\frac{\partial \omega^{ik}}{\partial x^l},
	\end{equation*}
we have that these two equations are true if and only if $\omega$ is symplectic. Finally, to check Eq. \eqref{eq:integrability7} we can modify the previous relation in order to obtain
	\begin{equation*}
	  \frac{\partial \omega^{ik}}{\partial x^l} = \omega^{ir}\omega^{ks}\frac{\partial \omega_{rs}}{\partial x^l}.
	\end{equation*}
Then, it can be checked that, as we expected, Eq. \eqref{eq:integrability7} is true if and only if $d\omega = 0$, that is, if $\omega$ is symplectic.
\end{example}

% SECTION: PROOF OF THEOREM 1.1
\section{Proof of Theorem \ref{thm:s6structures}}
\label{SECTION:PROOFMAINTHEOREM}

Before detailing the proof of Theorem \ref{thm:s6structures}, we must justify why we use certain $\fm$-linear combinations of the structures $\J_{1,\J}$, $\J_g$, $\J_\omega$ in order to find a weak generalized complex structure. The argument is grounded on the properties of commutation and anti-commutation of the previously introduced structures. As we are only working with endomorphisms $\J\colon \TM\to \TM$ with $\J^2 = -\mathcal Id$, we introduce the following concept.

\begin{definition}
\label{def:hypercomplexstructure}
A \emph{generalized almost hypercomplex structure} is a structure $(\J_1, \J_2, \J_3)$ formed by three weak generalized almost complex structures such that they anti-commute (that is, $\J_i\J_j = -\J_j\J_i$ for $i\neq j$) and $\J_3 = \J_2 \J_1$.
\end{definition}

When the base manifold is endowed with an almost Hermitian or an almost Norden structure, a generalized almost hypercomplex structure is naturally induced, as it is shown in the next result.

\begin{proposition}[{\cite[Prop. 5.4]{ETAYOGOMEZNICOLASSANTAMARIA2024}}]
\label{prop:triplemjg}
Let $(M, J, g)$ be a manifold and $\var\in \{-1, 1\}$ such that $J^2 = -Id$ and $g(JX, JY) = \var g(X, Y)$ for every $X, Y\in \xm$, and let $\omega(\cdot, \cdot) = g(J\cdot, \cdot)$. Then, $(\J_{\var, J}, \J_g, \J_\omega)$ is a generalized almost hypercomplex structure.
\end{proposition}

The anti-commutation of these three generalized structures is based on the fact that, for such a manifold $(M, J, g)$ with $J^2 = -Id$ and $g(JX, JY) = \var g(X, Y)$, the musical isomorphisms of $g$ and $\omega$ are related between them (see \cite[Prop. 2.8]{ETAYOGOMEZNICOLASSANTAMARIA2022}):
	\begin{gather*}
	  \flat_\omega = \flat_g J = -\var J^*\flat_g,   \\
	  \var \sharp_\omega = \sharp_g J^* = - \var J\sharp_g.
	\end{gather*}
Because this is the only generalized almost hypercomplex structure that can be generated from $(M, J, g)$ using the endomorphisms $\J_{\lambda,J}, \J_g, \J_\omega$ from Eqs. (\ref{eq:complexlambdaJ}-\ref{eq:complexg}) for $\lambda\in \{-1, 1\}$, the following corollary is easily inferred.

\begin{corollary}
\label{cor:sphericalcombination}
Let $(M, J, g)$ be an almost Hermitian manifold (i.e., $\var = 1$). Then, a linear combination $a\J_{\lambda,J} + b\J_g + c\J_\omega$ with $a,b,c\in \fm$ is a weak generalized almost complex structure if and only if $\lambda = 1$ and $a^2 + b^2 + c^2 = 1$. The structure is strong if and only if $c \equiv \pm 1$ (or, equivalently, if and only if $a = b \equiv 0$).
\end{corollary}

\begin{proof}
The value of $\lambda$ must be equal to $1$ because $g(JX, JY) = g(X, Y)$ for every $X, Y\in \xm$ and, from Proposition \ref{prop:triplemjg}, $\lambda = \var = 1$ and $(\J_{1, J}, \J_g, \J_\omega)$ is a generalized almost hypercomplex structure. Apart from that, we compute the square of any endomorphism obtained as a linear combination $\J = a\J_{1,J} + b\J_g + c\J_\omega$ with $a,b,c\in \fm$:
	\begin{equation*}
	  \begin{split}
	    \J^2 &= (a\J_{1,J} + b\J_g + c\J_\omega)^2  \\
	    	 &= a^2\J_{1,J}^2 + b^2\J_g^2 + c^2\J_\omega^2 + ab(\J_{1,J}\J_g + \J_g\J_{1,J}) + ac(\J_{1,J}\J_\omega + \J_\omega\J_{1,J}) + bc(\J_g\J_\omega + \J_\omega\J_g)  \\
	    	 &= -(a^2 + b^2 + c^2)\mathcal Id.
	  \end{split}
	\end{equation*}
Therefore, it must be $a^2 + b^2 + c^2 = 1$. To check whether $\J$ is strong or not, we calculate $\G_0(\J u, \J v)$ for any $u,v\in \GTM$:
	\begin{equation*}
	  \begin{split}
	    \G_0(\J u, \J v) &= \G_0(a\J_{1,J} u + b\J_g u + c\J_\omega u, a\J_{1,J} v + b\J_g v + c\J_\omega v)  \\
	    	 			 &= a^2\G_0(\J_{1,J} u, \J_{1,J} v) + b^2\G_0(\J_g u, \J_g v) + c^2\G_0(\J_\omega u, \J_\omega v) + ab(\G_0(\J_{1,J} u, \J_g v)  \\
	    	 			 &\phantom{=}  + \G_0(\J_g u, \J_{1,J} v)) + ac(\G_0(\J_{1,J} u, \J_\omega v) + \G_0(\J_\omega u, \J_{1,J} v)) + bc(\G_0(\J_g u, \J_\omega v) + \G_0(\J_\omega u, \J_g v))  \\
	    	 			 &= -a^2\G_0(u, v) - b^2\G_0(u, v) + c^2\G_0(u, v) + ab(\G_0(\J_\omega u, v) + \G_0(u, \J_\omega v))  \\
	    	 			 &\phantom{=} + ac(\G_0(u, \J_g v) + \G_0(\J_g u, v)) - bc(\G_0(\J_{1,J} u, v) + \G_0(u, \J_{1,J} v))  \\
	    	 			 &= (-a^2 - b^2 + c^2)\G_0(u, v) + 2ac\G_0(u, \J_g v) - 2bc\G_0(\J_{1,J} u, v).
	  \end{split}
	\end{equation*}
Thus, if the structure $\J$ is strong then $-a^2 - b^2 + c^2 = 1$. This condition combined with $a^2 + b^2 + c^2 = 1$ shows that it must be $c \equiv \pm 1$. The converse statement is immediate to check.
\end{proof}
We will call such a $\fm$-linear combination $a\J_{1,J} + b\J_g + c\J_\omega$ a \emph{spherical combination} of the generalized almost hypercomplex structure.

Our main goal is to ask ourselves if, for the nearly Kähler structure $(\Sph^6, J, g)$ inherited from the pure octonions product, there is any integrable spherical combination of the generalized almost hypercomplex structure $(\J_{1, J}, \J_g, \J_\omega)$. We will work in spherical coordinates for $\Sph^6\subset \R^7$: if we take angular coordinates $u^1,\dots, u^6$ such that $u^1,\dots,u^5\in(0,\pi)$ and $u^6\in (0,2\pi)$, then the coordinates of the six-dimensional sphere are the following:
	\begin{equation*}
	  \arraycolsep=1.4pt
	  \left\lbrace
	  \begin{array}{rl}
		x^1 &= \cos(u^1),  \\
		x^2 &= \sin(u^1)\cos(u^2),  \\
		x^3 &= \sin(u^1)\sin(u^2)\cos(u^3),  \\
		    & \vdots  \\
		x^6 &= \sin(u^1)\sin(u^2)\sin(u^3)\sin(u^4)\sin(u^5)\cos(u^6),  \\
		x^7 &= \sin(u^1)\sin(u^2)\sin(u^3)\sin(u^4)\sin(u^5)\sin(u^6).
	  \end{array}
	  \right.
	\end{equation*}
The metric $g$ on $\Sph^6$ is the one induced by the Euclidean metric on $\R^7$, where $\Sph^6$ is seen as a submanifold of the seven-dimensional Euclidean space. Therefore, in spherical coordinates this metric is given by
	\begin{equation}
	  \label{eq:metriclocal}
	  g_{ij} = \left\lbrace
	  \begin{array}{ll}
		1 								& \text{if }i = j = 1,  \\
		\sin^2(u^1)\dots\sin^2(u^{i-1}) & \text{if }i = j\neq 1,  \\
		0 								& \text{if }i \neq j.
	  \end{array}
	  \right.
	\end{equation}
On the other hand, the local representation of the inverse metric can be easily computed from the data in Eq. \eqref{eq:metriclocal}:
	\begin{equation}
	  \label{eq:inversemetriclocal}
	  g^{ij} = \left\lbrace
	  \begin{array}{ll}
		1 										  & \text{if }i = j = 1,  \\
		\dfrac{1}{\sin^2(u^1)\dots\sin^2(u^{i-1})} & \text{if }i = j\neq 1,  \\
		0 										  & \text{if }i \neq j.
	  \end{array}
	  \right.
	\end{equation}

The almost complex structure $J$ on $\Sph^6$ is inherited by the pure octonions product in $\R^7$. If we denote this product by $\times$, in any point $p\in \Sph^6$ the structure $J_p\colon T_p\Sph^6\to T_p\Sph^6$ is defined as
	\begin{equation*}
	  J_p w = p\times w,
	\end{equation*}
for any $w\in T_pM$. The explicit multiplication table can be consulted, for example, in \cite[Ch. 19]{CHEN2011}. As the explicit expression of $J$ in local coordinates is really long and it would not fit in one page, we will just give some specific values of $J^i_j$. In particular, we will use that
	\begin{equation}
	  \label{eq:Jpartial1}
	  J\partial_1 = \frac{\cos(u^3)}{\sin(u^1)}\partial_2 - \frac{\cos(u^2)\sin(u^3)}{\sin(u^1)\sin(u^2)}\partial_3 + \frac{\cos(u^5)}{\sin(u^1)}\partial_4 - \frac{\cos(u^4)\sin(u^5)}{\sin(u^1)\sin(u^4)}\partial_5 - \frac{1}{\sin(u^1)}\partial_6,
	\end{equation}
and
	\begin{equation}
	  \label{eq:J*du1}
	    \begin{split}
	      J^*du^1 &= -\sin(u^1)\cos(u^3) du^2 + \sin(u^1)\cos(u^2)\sin(u^2)\sin(u^3) du^3 - \sin(u^1)\sin^2(u^2)\sin^2(u^3)\cos(u^5) du^4  \\
	      		  &\phantom{=} + \sin(u^1)\sin^2(u^2)\sin^2(u^3)\cos(u^4)\sin(u^4)\sin(u^5) du^5 + \sin(u^1)\sin^2(u^2)\sin^2(u^3)\sin^2(u^4)\sin^2(u^5) du^6.
	    \end{split}
	\end{equation}
Therefore, as $J\partial_1 = J_1^k\partial_k$ and $J^*du^1 = J^1_kdu^k$, we have the expressions for each $J^k_1$ and $J^1_k$.

It can be checked that, for any $X,Y\in \mathfrak X(\Sph^6)$, we have $g(JX, JY) = g(X, Y)$. Therefore, Corollary \ref{cor:sphericalcombination} assures that any spherical spherical combination of the generalized structures $(\J_{1,J}, \J_g, \J_\omega)$ is a weak generalized almost complex structure.

Because the proof of Theorem \ref{thm:s6structures} is quite extensive, we firstly check what happens when the contribution of $\J_g$ in the spherical combination is null (i.e. when $b \equiv 0$).

\begin{proposition}
\label{prop:bequal0}
Let $(\Sph^6, J,\omega)$ be the six-dimensional sphere with its usual nearly Kähler structure. Then, there is not any spherical combination $\J = a\J_{1,J} + c\J_\omega$ with $a,c\in \mathfrak F(\Sph^6)$ and $a^2 + c^2 = 1$ such that the weak generalized almost complex structure $\J$ is integrable.
\end{proposition}

\begin{proof}
To see that such a weak generalized almost complex structure $\J$ cannot be integrable, we will analyse the necessary conditions shown in Theorem \ref{thm:integrabilitycoord}. In particular, we will use both Eqs. (\ref{eq:integrability3}, \ref{eq:integrability5}). Firstly, we interchange the indices $i,j$ in Eq. \eqref{eq:integrability5}, obtaining the following expression:
	\begin{equation*}
	  B^{jk}\frac{\partial A^l_i}{\partial x^k} - A^k_i\frac{\partial B^{jl}}{\partial x^k} + A^l_k\frac{\partial B^{jk}}{\partial x^i} + B^{kl}\left(\frac{\partial D^j_k}{\partial x^i} - \frac{\partial D^j_i}{\partial x^k}\right) = 0.
	\end{equation*}
This equation is quite similar to Eq. \eqref{eq:integrability3}; in fact, if both expressions are added up some terms will vanish, resulting in the following identity:
	\begin{equation*}
	  B^{kl}\left(\frac{\partial A_i^j}{\partial x^k} + \frac{\partial D_i^j}{\partial x^k}\right) = 0.
	\end{equation*}
This condition must be fulfilled by any integrable structure $\J$ with $\J^2 = -\mathcal Id$. Taking now the spherical combination $\J = a\J_{1,J} + c\J_\omega$ and using its matrix form, it is $A = aJ$, $B = -c\sharp_\omega = cJ\sharp_g$, $C = c\flat_\omega = c\flat_g J$ and $D = aJ^*$. The local representations of these morphisms are
	\begin{equation*}
	  A_i^j = aJ_i^j, \enspace\enspace B^{ij} = cg^{ik}J_k^j, \enspace\enspace C_{ij} = cJ_i^k g_{kj}, D^i_j = aJ^i_j,
	\end{equation*}
so the previous condition is turned into
	\begin{equation*}
	  cg^{ks}J_s^l\frac{\partial}{\partial u^k}\left(a J^j_i\right) = 0.
	\end{equation*}
If this equation is developed, it can be written in the following form:
	\begin{equation}
	  \label{eq:equationproofb=0}
	  ac\left(g^{ks}J_s^l\frac{\partial J^j_i}{\partial u^k}\right) + c\frac{\partial a}{\partial u^k}\left(g^{ks}J_s^l J_i^j\right) = 0.
	\end{equation}
	
In this last equation, there are two clearly differentiated parts: one related to the product $ac$, and other related to the product of $c$ with each derivative of $a$ with respect to $u^k$. The idea is to compare Eq. \eqref{eq:equationproofb=0} for different values of the indices $i,j,k$ and, by combining them, to obtain restraints for the functions $a$ and $c$. We use the combinations of indices $i=l=1$, $j=2$; and $i=l=1$, $j=3$:
\begin{itemize}

\item $i=l=1$, $j=2$:
	\begin{equation}
	  \label{eq:b=0il1j2}
	  ac\left(g^{ks}J_s^1\frac{\partial J^2_1}{\partial u^k}\right) + c\frac{\partial a}{\partial u^k}\left(g^{ks}J_s^1 J_1^2\right) = 0.
	\end{equation}
	
\item $i=l=1$, $j=3$:
	\begin{equation}
	  \label{eq:b=0il1j3}
	  ac\left(g^{ks}J_s^1\frac{\partial J^3_1}{\partial u^k}\right) + c\frac{\partial a}{\partial u^k}\left(g^{ks}J_s^1 J_1^3\right) = 0.
	\end{equation}

\end{itemize}
From Eq. \eqref{eq:Jpartial1}, we know that $J_1^2 = \frac{\cos(u^3)}{\sin(u^1)}$ and $J_1^3 = -\frac{\cos(u^2)\sin(u^3)}{\sin(u^1)\sin(u^2)}$. We can multiply Eq. \eqref{eq:b=0il1j2} by $-J_1^3$ and Eq. \eqref{eq:b=0il1j3} by $J_1^2$ and add them up, obtaining
	\begin{equation*}
	  ac\left[g^{ks}J_s^1\left(J_1^2\frac{\partial J_1^3}{\partial u^k} - J_1^3\frac{\partial J_1^2}{\partial u^k}\right)\right] = 0.
	\end{equation*}
If the sum over the indices $k,s$ explicitly, as $J_1^2, J_1^3$ only depend on $u^1,u^2,u^3$, and knowing that $g^{ij} = 0$ for $i\neq j$ and $J^1_1 = 0$, the following expression is inferred:
	\begin{equation*}
	  \begin{split}
	    0 &= ac\left[g^{1s}J_s^1\left(J_1^2\frac{\partial J_1^3}{\partial u^1} - J_1^3\frac{\partial J_1^2}{\partial u^1}\right) + g^{2s}J_s^1 J_1^2\frac{\partial J_1^3}{\partial u^2} + g^{3s}J_s^1\left(J_1^2\frac{\partial J_1^3}{\partial u^3} - J_1^3\frac{\partial J_1^2}{\partial u^3}\right)\right]  \\
	      &= ac\left[g^{22}J_2^1 J_1^2\frac{\partial J_1^3}{\partial u^2} + g^{33}J_3^1\left(J_1^2\frac{\partial J_1^3}{\partial u^3} - J_1^3\frac{\partial J_1^2}{\partial u^3}\right)\right].
	  \end{split}
	\end{equation*}
We calculate explicitly the function that is multiplying $ac$, taking from Eq. \eqref{eq:J*du1} that $J^1_2 = -\sin(u^1)\cos(u^3)$ and $J^1_3 = \sin(u^1)\cos(u^2)\sin(u^2)\sin(u^3)$:
	\begin{equation*}
	  \begin{split}
	  g^{22}J_2^1 J_1^2\frac{\partial J_1^3}{\partial u^2} &= -\frac{\sin(u^1)\cos(u^3)}{\sin^2(u^1)}\frac{\cos(u^3)}{\sin(u^1)}\frac{\sin(u^3)}{\sin(u^1)\sin^2(u^2)} = -\frac{\cos^2(u^3)\sin(u^3)}{\sin^3(u^1)\sin^2(u^2)},  \\
	  g^{33}J_3^1 J_1^2\frac{\partial J_1^3}{\partial u^3} &= -\frac{\sin(u^1)\cos(u^2)\sin(u^2)\sin(u^3)}{\sin^2(u^1)\sin^2(u^2)}\frac{\cos(u^3)}{\sin(u^1)}\frac{\cos(u^2)\cos(u^3)}{\sin(u^1)\sin(u^2)} = -\frac{\cos^2(u^2)\cos^2(u^3)\sin(u^3)}{\sin^3(u^1)\sin^2(u^2)},  \\
	  g^{33}J_3^1 J_1^3\frac{\partial J_1^2}{\partial u^3} &= \frac{\sin(u^1)\cos(u^2)\sin(u^2)\sin(u^3)}{\sin^2(u^1)\sin^2(u^2)}\frac{\cos(u^2)\sin(u^3)}{\sin(u^1)\sin(u^2)}\frac{\sin(u^3)}{\sin(u^1)} = \frac{\cos^2(u^2)\sin^3(u^3)}{\sin^3(u^1)\sin^2(u^2)}.
	  \end{split}
	\end{equation*}
Joining all together,
	\begin{equation*}
	  -ac\frac{\sin(u^3)(\cos^2(u^3) + \cos^2(u^2))}{\sin^3(u^1)\sin^2(u^2)} = 0.
	\end{equation*}
Therefore, it must be $a\equiv 0$ or $c\equiv 0$. However, that would imply $\J = \pm\J_\omega$, which is not integrable because $\omega$ is not a symplectic form; or $\J = \pm\J_{1,J}$, which is also not integrable because $J$ is not integrable. In consequence, there are no functions $a,c\in \mathfrak F(\Sph^6)$ such that $a^2 + c^2 = 1$ and $a\J_{1,J} + c\J_\omega$ is a weak generalized complex structure.
\end{proof}

After proving Proposition \ref{prop:bequal0}, we are in a good position to prove the general case in Theorem \ref{thm:s6structures}, taking any possible expression for the function $b$.

\begin{proof}[Proof of Theorem \ref{thm:s6structures}]
We take a spherical combination $\J = a\J_{1,J} + b\J_g + c\J_\omega$ with $a,b,c\in \mathfrak F(\Sph^6)$. We will work mainly with Eq. \eqref{eq:integrability7}; specifically, we symmetrise the expression with respect to the indices $i,j$. To do this, we interchange the indices $i,j$ in Eq. \eqref{eq:integrability7}:
	\begin{equation*}
	  B^{jk}\frac{\partial B^{il}}{\partial x^k} - B^{ik}\frac{\partial B^{jl}}{\partial x^k} - B^{kl}\frac{\partial B^{ji}}{\partial x^k} = 0.
	\end{equation*}
We add this formula to Eq. \eqref{eq:integrability7}, obtaining the following expression in local coordinates:
	\begin{equation}
	  \label{eq:symmetricequationB}
	  B^{kl}\frac{\partial}{\partial x^k}(B^{ij} + B^{ji}) = 0.
	\end{equation}
This condition must be satisfied for a weak generalized complex structure. Using the matrix notation from Eq. \eqref{eq:matrixendomorphism} for the almost complex structure $\J$, it must be $B = -(b\sharp_g + c\sharp_\omega)$. Taking local coordinates, as $\sharp_\omega = -J\sharp_g = \sharp_g J^*$, the local representation for $B$ is
	\begin{equation*}
	  B^{ij} = cg^{ik}J_k^j - bg^{ij} = -cg^{jk}J_k^i -  bg^{ij}.
	\end{equation*}
Then, making substitutions in Eq. \eqref{eq:symmetricequationB}, we have that the following condition must be met:
	\begin{equation}
	  \label{eq:explicitsymmetricequationB}
	  bc\left(g^{ks}J_s^l\frac{\partial g^{ij}}{\partial u^k}\right) - b^2\left(g^{kl}\frac{\partial g^{ij}}{\partial u^k}\right) + c\frac{\partial b}{\partial u^k}\left(g^{ks}J_s^lg^{ij}\right) - b\frac{\partial b}{\partial u^k}\left(g^{kl}g^{ij}\right) = 0.
	\end{equation}

We want now to compare the expression in Eq. \eqref{eq:explicitsymmetricequationB} for different values of $i,j,l$ and hopefully find any condition for the functions $b,c$. We work in spherical coordinates and use the explicit local expressions for $g$ and $J$; in particular, we will use the values $g^{11} = 1 $, $g^{22} = \frac{1}{\sin^2(u^1)}$ and $J^1_1 = 0$. Eq. \eqref{eq:explicitsymmetricequationB} is compared for the indices $i=j=l=1$; and $i=j=2$, $l=1$:
\begin{itemize}

\item $i=j=l=1$:
	\begin{equation*}
	  \begin{split}
	    0 &= bc\left(g^{ks}J_s^1\frac{\partial g^{11}}{\partial u^k}\right) - b^2\left(g^{k1}\frac{\partial g^{11}}{\partial u^k}\right) + c\frac{\partial b}{\partial u^k}\left(g^{ks}J_s^1g^{11}\right) - b\frac{\partial b}{\partial u^k}\left(g^{k1}g^{11}\right)  \\
	   	  &= c\frac{\partial b}{\partial u^k}\left(g^{ks}J_s^1\right) - b\frac{\partial b}{\partial u^1}.
	  \end{split}
	\end{equation*}

\item $i=j=2$, $l=1$:
	\begin{equation*}
	  \begin{split}
	    0 &= bc\left(g^{ks}J_s^1\frac{\partial g^{22}}{\partial u^k}\right) - b^2\left(g^{k1}\frac{\partial g^{22}}{\partial u^k}\right) + c\frac{\partial b}{\partial u^k}\left(g^{ks}J_s^1g^{22}\right) - b\frac{\partial b}{\partial u^k}\left(g^{k1}g^{22}\right)  \\
	   	  &= bc\left(g^{11}J_1^1\frac{\partial g^{22}}{\partial u^1}\right) - b^2\left(g^{11}\frac{\partial g^{22}}{\partial u^1}\right) + c\frac{\partial b}{\partial u^k}\left(g^{ks}J_s^1g^{22}\right) - b\frac{\partial b}{\partial u^1}\left(g^{11}g^{22}\right)  \\
	   	  &= \frac{2\cos(u^1)}{\sin^3(u^1)}b^2 + \frac{1}{\sin^2(u^1)}\left[c\frac{\partial b}{\partial u^k}\left(g^{ks}J_s^1\right) - b\frac{\partial b}{\partial u^1}\right].
	  \end{split}
	\end{equation*}

\end{itemize}

Taking both expressions together, it is immediate to see that
	\begin{equation*}
	  \frac{2\cos(u^1)}{\sin^3(u^1)}b^2 = 0,
	\end{equation*}
so we have that $b \equiv 0$. Then, in order to be integrable, the structure must be $\J = a\J_{1,J} + c\J_\omega$ with $a^2 + c^2 = 1$. However, Proposition \ref{prop:bequal0} states that any such structure can be integrable. Therefore, there are no functions $a,b,c\in \mathfrak F(\Sph^6)$ such that $a^2 + b^2 + c^2 = 1$ and $a\J_{1,J} + b\J_g + c\J_\omega$ is a weak generalized complex structure, thus proving the theorem.
\end{proof}

% BIBLIOGRAPHY

\end{document}